\documentclass[11pt,reqno]{amsart}
\usepackage{fullpage}
\usepackage{cite}
\usepackage[msc-links,initials]{amsrefs}

\usepackage[english]{babel}
\usepackage[utf8x]{inputenc}
\usepackage{amsmath, amssymb, amsthm, mathrsfs, graphicx, float}

\usepackage[colorinlistoftodos]{todonotes}
\usepackage{tikzsymbols}
\usepackage{mathtools}
\usepackage{hyperref}
\usepackage{dsfont}
\usepackage{nicefrac}
\usepackage{tikz-cd}

\DeclarePairedDelimiter\floor{\lfloor}{\rfloor}

\newtheorem{theorem}{Theorem}[section]
\newtheorem{lemma}[theorem]{Lemma}
\newtheorem{proposition}[theorem]{Proposition}

\newtheorem{corollary}[theorem]{Corollary}
\theoremstyle{remark}

\newtheorem{remark}[theorem]{Remark}
\theoremstyle{definition}
\newtheorem{definition}[theorem]{Definition}
\newtheorem{Conjecture}[theorem]{Conjecture}

\newcommand{\N}{\mathbb{N}}

\newcommand{\F}{\mathbb{F}}
\newcommand{\I}{\mathcal{I}_q}
\newcommand{\M}{\mathcal{M}}
\renewcommand{\P}{\mathcal{P}}
\newcommand{\e}{\varepsilon}
\newcommand{\ud}{\overline{d}}
\newcommand{\ld}{\underline{d}}

\newcommand{\m}{\rule[0.25em]{4pt}{0.5pt}}
\newcommand{\ignore}[1]{}

\title{On the Size of Primitive Sets in Function Fields}
\author{Andr\'{e}s G\'{o}mez-Colunga}
\address{Department of Mathematics, Yale University, New Haven, CT 06520}
\email{andres.gomez-colunga@yale.edu}
\author{Charlotte Kavaler}
\address{Department of Mathematics, Yale University, New Haven, CT 06520}
\email{charlotte.kavaler@yale.edu}
\author{Nathan McNew} 
\address{Department of Mathematics, Towson University, Towson, MD 21252}
\email{nmcnew@towson.edu}
\author{Mirilla Zhu}
\address{Department of Mathematics, Yale University, New Haven, CT 06520}
\email{mirilla.zhu@yale.edu}
\date{}

\begin{document}

\begin{abstract}
    A set is primitive if no element of the set divides another.  We consider primitive sets of monic polynomials over a finite field and find natural generalizations of many of the results known for primitive sets of integers. In particular, we show that primitive sets in the function field have lower density zero by showing that the sum $\sum_{a \in A} \frac{1}{q^{\deg a}\deg a}$, an analogue of a sum considered by Erd\H{o}s, is uniformly bounded over all primitive sets $A$. We then adapt a method of Besicovitch to construct primitive sets in $\mathbb{F}_q[x]$ with upper density arbitrarily close to $\frac{q - 1}{q}$ and generalize a result of Martin and Pomerance on the asymptotic growth rate of the counting function of a primitive set. Along the way we prove a quantitative analogue of the Hardy-Ramanujan theorem for function fields, as well as bounds on the size of the $k$-th irreducible polynomial. 
\end{abstract}

\maketitle

\section{Introduction}
We call a set primitive if no element of the set divides another. A natural question to ask of primitive sets is how large they can be. In the integers, this has been answered for various notions of size. Erd\H{o}s \cite{erdos35} proved in 1935 that primitive sets have lower density zero by showing that \begin{equation} \sum_{n \in A} \frac{1}{n \log n}\label{eq:erdsum} \end{equation}
converges for any primitive set $A\neq \{1\}$.  In fact, Erd\H{o}s showed that \eqref{eq:erdsum} is uniformly bounded and later conjectured that it is maximized when the set $A$ is taken to be the prime numbers. 
That same year, Besicovitch \cite{besicovitch} showed that there exist primitive sets $A$ with upper density greater than  $\frac{1}{2}-\epsilon$ for any $\epsilon>0$. Because there cannot exist a primitive set having upper density greater than or equal to $\frac{1}{2}$, this is essentially the best possible upper density of a primitive set. 

In contrast to the occasionally large but usually very sparse sets in Besicovitch's construction, Ahlswede, Khachatrian, and S\'{a}rk\"{o}zy \cite{aks99} consider primitive sets that are large and consistently growing. They show that there exists a primitive set $A$ whose counting function $\mathcal{A}(x) = |[1,x] \cap A|$ satisfies \[\mathcal{A}(x) \gg \frac{x}{\log \log x (\log \log \log x)^{1+\epsilon}}\] for any $\epsilon>0$. This is nearly best possible, as one can show that $\mathcal{A}(x) = o\left(\frac{x}{\log \log x (\log \log \log x)}\right)$ for any primitive set.

Nevertheless, Martin and Pomerance make a small improvement regarding the $\epsilon$ in this result  \cite{martin-pomerance}. They prove that given any positive increasing function $L(x)$ satisfying $L(x) \sim L(2x)$ such that $\int_2^\infty \frac{dt}{t\log t L(t)} <\infty$, there exists a primitive set $A$ with counting function that satisfies
\[\mathcal{A}(x) \asymp \frac{x}{\log \log x \cdot \log \log \log x \cdot L(\log \log x)}\] for sufficiently large $x$.  In particular one can take $L(x)=\log_2 x \log_3 x \cdots (\log_{j-1} x)^{1+\epsilon}$, (where $\log_k x$ denotes the $k$-fold iterated logarithm) for any $j>2$ to obtain a primitive set with counting function
\[\mathcal{A}(x) \asymp \frac{x}{\log_2 x \cdot \log_3 x \cdots \log_j x \cdot  (\log_{j+1} x)^{1+\epsilon}}.\]

In this paper, we consider primitive subsets of polynomials over a finite field $\mathbb{F}_q[x]$ and obtain results analogous to those of Erd\H{o}s, Besicovitch, and Martin and Pomerance in this setting. In Corollary \ref{cor:low}, we show that the lower density of a primitive set in the function field is always zero by considering a sum analogous to \eqref{eq:erdsum}. We then give a construction in Theorem \ref{thm:besicov} of a set with upper density arbitrarily close to $\frac{q-1}{q}$. Finally, we prove a function field analogue of Martin and Pomerance's result in Theorem \ref{thm.pom}, demonstrating the existence of primitive sets $S \subset \F_q[x]$ with consistently growing counting functions ${S'(n) = |\{f \in S: \deg(f) = n\}|}$ of size
\[S'(n) \asymp \frac{q^n}{\log n \cdot \log \log n \cdot L(\log n)}\]
where $L(n)$ satisfies the same restrictions as in the integer case. 

In order to prove this result, we establish two additional results which may be of independent interest.  For any ordering $\{P_k\}$ of the monic irreducible polynomials over $\F_q$ such that $\deg P_k \leq \deg P_j$ if $k<j$ we show in Theorem \ref{thm:kthirred} that \[\log_q k + \log_q(\log_q k) + \log_q(q-1) -1 +  o(1) \ \leq \ \deg P_k \ \leq \ \log_q k + \log_q(\log_q k) + \log_q(q-1) +  o(1)\]
as $k\to \infty$. Then, in Proposition \ref{prop.quant-hr-2}, we obtain bounds on the number of polynomials of degree $n$ having either unusually few or unusually many irreducible factors.  In particular, for fixed $0<\alpha <1<\beta$ we determine that the number of degree $n$ polynomials having at most $\alpha \log n$ or at least $\beta \log n$ factors is $O_\alpha\left(\frac{q^n}{n^{Q(\alpha)}\sqrt{\log n}}\right)$ and  $O_\beta\left(
 \frac{q^n}{n^{Q(\beta)}\sqrt{\log n}}\right)$ respectively, where $Q(y) = y\log y - y + 1$.

\subsection{Primitive sets of polynomials} 
Let $\mathcal{M}_q$ denote the set of monic polynomials in $\F_q[x]$ and $\mathcal{I}_q$ denote the set of irreducible polynomials in $\mathcal{M}_q$. Just as in the integers, we say that a set $A \subset \M_q$ is primitive if no element divides another. It is not difficult to see that $\mathcal{I}_q$ is primitive; some other examples of primitive sets include the set of monic polynomials with exactly $k$ irreducible factors counted with multiplicity and the set of monic polynomials of degree $n$. In this paper, we will compare the growth of primitive sets through two measures of size: counting functions and natural densities. 
\begin{definition}
For $S \subset \mathbb{F}_q[x]$, the counting functions of $S$ are given by $S(n) = \#\{f \in S : \deg f \leq n\}$ and $S'(n) = \#\{f \in S : \deg f = n\}$. 
\end{definition}
\begin{definition}
The \textit{natural density} of $S$ is given by
\[d(S) = \lim_{n \to \infty} \frac{S(n)}{\M_q(n)},\]
and the \textit{upper density}  and \textit{lower density} of $S$ are given by 
\[\overline{d}(S) = \limsup_{n \to \infty} \frac{S(n)}{\M_q(n)} \hspace{8mm} \textnormal{and}\hspace{8mm} \underline{d}(S) = \liminf_{n \to \infty} \frac{S(n)}{\M_q(n)}. \]
\end{definition}

\section{Lower Density of Primitive Sets}
Intuitively, a primitive set cannot be too large because including any one element in the set means that all multiples of that element must be excluded. Here, we formalize this notion by showing that the lower density of any primitive set $A \subset \mathcal{M}_q$ must be zero. In order to do so, we generalize a 1935 proof of Erd\H{o}s \cite{erdos35} to the function field setting. Since the result is immediate for $A = \{1\}$, we will assume $A \neq \{1\}$ for the remainder of this section.

Following Erd\H{o}s, \cite{erdos35} our proof depends on the convergence of the sum 
\begin{equation} \sum_{a \in A} \frac{1}{\|a\| \deg a},\label{eq:erdffsum} \end{equation}
for all primitive $A \subset \M_q$, where $\|a\| := q^{\deg a}$ denotes the norm of $a$. Because this sum is a function field analogue of the sum \eqref{eq:erdsum} Erd\H{o}s considered in his 1935 paper, we will call this sum the Erd\H{o}s sum of $A$.

For any polynomial $f \in \F_q[x]$, let $d(f)$ denote the smallest degree of an irreducible factor of $f$, and let $D(f)$ denote the largest degree of an irreducible factor of $f$. By the Sieve of Eratosthenes, the density of the set $\left\{g \in \M_q : f|g,\ \! d\left(\frac{g}{f}\right)>D(f)\right\}$ of polynomials which are multiples of $f$ but are not divisible by any other polynomials of degree less than $D(f)$ is 
\begin{equation} \frac{1}{\|f\|} \ \prod_{\substack{p\in \mathcal{I}_q \\ \deg p \leq D(f)}} \left(1 - \frac{1}{\|p\|}\right). \label{eq:densmults} \end{equation}
This can be derived from the fact that for any fixed polynomial $p$, the density of multiples of $p$ is exactly $\frac{1}{\|p\|}$. We start by showing that the sum of densities of the form \eqref{eq:densmults} ranging over all elements $f$ contained in a primitive set $A$ is no greater than 1.
\begin{proposition}\label{prop.erdossum}
If $A$ is a primitive set, then 
\[\sum_{a \in A} \left( \frac{1}{\|a\|} \ \prod_{\substack{p\in \mathcal{I}_q \\ \deg p \leq D(a)}} \left(1 - \frac{1}{\|p\|}\right) \right) \leq 1.\]
\end{proposition}

\begin{proof}
Suppose for contradiction that the inequality is false.\! Then there exists some $N\!\! \in\! \mathbb{N}$ such that\!\!

\vspace{-0.6em}
\[\sum_{\substack{a \in A \\ \deg a \leq N}} \left( \frac{1}{\|a\|} \ \prod_{\substack{p\in \mathcal{I}_q \\ \deg p \leq D(a)}} \left(1 - \frac{1}{\|p\|}\right) \right) > 1.\]

For any $n \in \mathbb{N}$ and 
$a \in A$, we define $a_n$ to be the number of monic polynomials of degree $n$ divisible by $a$ but by no $b \in A$ such that $D(b) \leq D(a)$. Note that included in the count $a_n$ are all degree $n$ polynomials of the form $ga$ with $d(g) > D(a)$. 

There are $q^{n-\deg a}$ polynomials of degree $n-\deg a$, and using the Sieve of Erastosthenes, we see that the number of such polynomials $g$ is approximately 

\[q^{n - \deg a}\prod_{\substack{p\in \mathcal{I}_q \\ \deg p \leq D(a)}} \left(1 - \frac{1}{\|p\|}\right) = q^{n - \deg a}\prod_{\substack{p\in \mathcal{I}_q \\ \deg p \leq D(a)}} \left(1 - \frac{1}{q^{\deg p}}\right),\] since each term in the product represents the proportion of polynomials not divisible by an irreducible polynomial $p$. If we choose $n$ large enough so that 
\[n \geq \deg a + \sum_{\mathclap{\substack{p \in \mathcal{I}_q \\ \deg p \leq D(a)}}} \deg p,\]
then the error in this approximation vanishes because all terms in the product expansion become integers. Hence for sufficiently large $n$,
\[a_n = q^{n - \deg a}\prod_{\substack{p\in \mathcal{I}_q \\ \deg p \leq D(a)}} \left(1 - \frac{1}{\|p\|}\right). \]

For any two polynomials $a', a'' \in A$, the polynomials counted by $a'_n$ and $a''_n$ form disjoint sets: the least irreducible factor of each polynomial counted by $a'_n$ and $a''_n$ has degree $\deg a'$ and $\deg a''$, respectively, and if $\deg a' = \deg a''$, then each polynomial counted by $a'_n$ will be divisible by $a'$, while no polynomial counted by $a''_n$ will be divisible by $a'$. Hence, we can sum over elements of $A$ to obtain
\[q^n \geq \sum_{\substack{a \in A \\ \deg a \leq N}} a_n
\geq \sum_{\substack{a \in A \\ \deg a \leq  N}} q^{n - \deg a}
\prod_{\substack{p\in \mathcal{I}_q \\ \deg p \leq D(a)}} \left(1 - \frac{1}{\|p\|}\right).\]
Dividing by $q^n$ gives 
\[1 \geq \sum_{\substack{a \in A \\ \deg a \leq N}}\left( \frac{1}{\|a\|}
\prod_{\substack{p\in \mathcal{I}_q \\ \deg p \leq D(a)}} \left(1 - \frac{1}{\|p\|}\right)\right),\]
which is a contradiction since we assumed that the sum on the right hand side was strictly greater than 1. Thus our original assumption must have been false, and so 
\[\sum_{a \in A} \left( \frac{1}{\|a\|} \ \prod_{\substack{p\in \mathcal{I}_q \\ \deg p \leq D(a)}} \left(1 - \frac{1}{\|p\|}\right) \right) \leq 1. \qedhere \]
\end{proof}

An analogue of Mertens' third theorem in function fields gives an asymptotic expression for the product in this expression.
\begin{theorem}\label{thm.mertens}
\[\prod_{\substack{p \in \mathcal{I}_q \\ \deg p \leq n}} \left(1 - \frac{1}{\|p\|} \right) \sim \frac{1}{e^{\gamma}n},\]
where $\gamma$ is the Euler-Mascheroni constant.
\end{theorem}
\begin{proof}
The result is a special case of Theorem 3 in \cite{MR1700882}.
\end{proof}

Using this result, we can show that the Erd\H{o}s sum \eqref{eq:erdffsum} of any primitive set $A$ converges, and in fact is uniformly bounded.

\begin{theorem}\label{thm:erdsum}
There exists a constant $C$ such that
\[\sum_{a \in A} \frac{1}{\|a\|\deg a} \leq C\]
for all primitive sets $A \subset \M_q$.
\end{theorem}
\begin{proof}
By Theorem \ref{thm.mertens}, there exists a constant $c$ such that 
\[\prod_{\substack{p \in \mathcal{I}_q \\ \deg p \leq n}} \left(1 - \frac{1}{\|p\|} \right) \geq \frac{c}{e^\gamma n} \] 
for all positive integers $n$. Hence for any $a \in A$,
\[\frac{1}{\|a\|} \ \prod_{\substack{p \in \mathcal{I}_q \\ \deg p \leq D(a)}} \left(1 - \frac{1}{\|p\|}\right) \geq \frac{c}{e^\gamma\|a\|D(a)} \geq \frac{c}{e^\gamma \|a\|\deg a}.\] Summing over all $a \in A$ gives us 

\[1 \geq \sum_{a \in A} \frac{1}{\|a\|} \ \prod_{\substack{p \in \mathcal{I}_q \\ \deg p \leq D(a)}} \left(1 - \frac{1}{\|p\|}\right) 
\geq \frac{c}{e^\gamma}\sum_{a \in A} \frac{1}{\|a\| \deg a},\]
upon which we see that 
\[\sum_{a \in A} \frac{1}{\|a\|\deg a} \leq \frac{e^\gamma}{c}.\qedhere\]
\end{proof}
Just as in the integer case, one can treat the Erd\H{o}s sum as a measure of the size of a primitive set which  gives larger weight to polynomials of lower degree; we explore this idea further in a future paper.  For now, we use this result to find the lower density of any primitive set.

 \begin{corollary} \label{cor:low}
If  $A \subset \M_q$ is any primitive set then  $\underline{d}(A) = 0.$
\end{corollary} 

\begin{proof}
Suppose for contradiction that $\ld(A) \not = 0$. Then there exists a positive constant $C$ so that

$$A(n) \geq \frac{C(q^{n+1} -1 )}{q-1} \geq Cq^n$$
for all sufficiently large $n$.  Using this and partial summation we find that  
\begin{align*}
\sum_{\substack{a\in A\\ \deg a \leq n}}\frac{1}{\|a\|\deg a} &\gg \log n. 
\end{align*}
This contradicts Theorem \ref{thm:erdsum}, so we can conclude that $\underline{d}(A) = 0$.
\end{proof}
It follows immediately that if the natural density of a primitive subset of $\M_q$ exists, then it must be equal to zero. The upper densities of such a primitive sets can be much greater however, as we see in the next section.

\section{Primitive Sets with Optimal Upper Density}
Primitive sets of integers cannot have upper density larger than $\frac{1}{2}$; this can be seen by partitioning the integers into disjoint sets, each of which contains an odd number and all of its even multiples. Since a primitive set can include at most one element from each of these sets, its upper density can be at most $\frac{1}{2}$.

A similar argument can be used in the function field. For some irreducible polynomial $f$ in $\M_q$ of degree 1, we partition $\M_q$ into disjoint subsets of the form $\{f^k g : k \in \N\}$, where $g$ is monic and not divisible by $f$. Like in the integer case, any primitive set $A$ can include at most one element from each of these sets, which gives a maximum upper density of $\frac{q-1}{q}$. Here, we demonstrate the existence of primitive sets with density arbitrarily close to this bound by generalizing Besicovitch's construction from \cite{besicovitch}. 

\begin{theorem} \label{thm:besicov}
For any $\epsilon > 0$, there exists a primitive set $A \subset \F_q[x]$ with 
\[\ud(A) > \frac{q-1}{q} - \e.\]
\end{theorem}
\begin{proof}
We construct a primitive set $A$ from an increasing sequence of positive integers $\{n_i\}$, to be determined shortly, as follows.
We include in $A$ all the monic polynomials of degree $n_1$, and note that this set contains 
\[\frac{\mathcal{M}_q'(n_1)}{\mathcal{M}_q(n_1)} = \frac{\mathcal{M}_q'(n_1)}{q^{n_1} + \ldots + q + 1} = \frac{q^{n_1}}{(q^{n_1+1}-1)/(q-1)} > \frac{q-1}{q}\]
of the monic polynomials of degree up to $n_1$. We then include all the monic polynomials of degree $n_2$, removing any polynomials having a divisor of degree $n_1$ so that our set remains primitive, and repeat this process for all $n_i$ to construct an infinite primitive set. Letting $I_n$ denote the monic polynomials of degree $n$ and $T_n$ denote the set of non-unit multiples of polynomials in $I_n$, we see that our primitive set can be written as $A = \cup_{i=1}^\infty (I_{n_i} \setminus \cup_{j=1}^{i-1} T_{n_j})$. 

We now define our sequence $\{n_i\}$ to ensure that the proportion of polynomials thrown out at each step is sufficiently small. For a given $\e > 0$, we require $\{n_i\}$ to satisfy two conditions: $\ud(T_{n_i}) \leq \frac{\e}{2^{i+1}}$, and 
$\frac{T_{n_{i-1}}(n)}{\mathcal{M}_q(n)} \leq \frac{\e}{2^i}$ for all $n \geq n_i$. 

We construct this sequence inductively. Car shows \cite{car84} that
$\lim_{n \to \infty} \ud(T_n) = 0$,
so there exists an integer $n_1$ such that $\ud(T_{n_1}) \leq \frac{\e}{2^2}$. Now suppose we have already found $n_1, n_2, ..., n_{j-1}$ satisfying the conditions of the sequence. Then since ${\ud(T_{n_{j-1}}) \leq \frac{\e}{2^j}}$, there exists $N_j$ such that for all $n \geq N_j$ 
\[\frac{T_{n_{j-1}}(n)}{\mathcal{M}_q(n)}\leq \frac{\e}{2^{j-1}}\]
by the definition of upper density. Furthermore, there exists $N_j'$ such that $\ud(T_n) \leq \frac{\e}{2^{j+1}}$ for all $n \geq N_j'$. Then if we let $n_j = \max(n_{j-1}+1, N_j, N_j')$, we see that $n_j$ satisfies all the conditions of the sequence.

We now show that for each $n_i$, the proportion of monic polynomials of degree up to $n_i$ which are in $A$ is at least $\frac{q-1}{q} - \e$:\\ 
\begin{align*}
    \frac{A(n_i)}{\mathcal{M}_q(n_i)} \geq
    \frac{|I_{n_i} \backslash \cup_{j = 1}^{i-1}T_{n_j})|}{\mathcal{M}_q(n_i)}  
    = \frac{|I_{n_i}| - |\cup_{j = 1}^{i-1}(T_{n_j} \cap I_{n_i})|}{\mathcal{M}_q(n_i)} \geq \frac{\mathcal{M}_q'(n_i)}{\mathcal{M}_q(n_i)} - \sum_{j=1}^{i-1} \frac{T_{n_j}(n_i)}{\mathcal{M}_q(n_i)}.
\end{align*}
We know $\frac{\mathcal{M}_q'(n_i)}{\mathcal{M}_q(n_i)} > \frac{q-1}{q}$, and furthermore, 
\[\sum_{j=1}^{i-1} \frac{T_{n_j}(n_i)}{\mathcal{M}_q(n_i)} \leq \sum_{j=1}^{i-1} \frac{\e}{2^i} < \e,\]
which implies
\[ \frac{A(n_i)}{\mathcal{M}_q(n_i)} \geq \frac{q-1}{q} - \e.\]
Because this is true for all $n_i$, we have
\[\ud(A) = \limsup_{n\to\infty}\frac{A(n)}{\mathcal{M}_q(n)} \geq \frac{q-1}{q} - \e. \qedhere\]
\end{proof}

\section{The size of the $k$-th irreducible polynomial}
Having investigated the natural densities of primitive sets, we now consider more carefully their counting functions. In particular, we construct primitive sets with consistently large counting functions, in contrast to the erratically growing counting functions of our modified Besicovitch construction. 

The asymptotic growth rate of the primitive sets we construct will be closely related to the distribution of irreducible polynomials in $\F_q[x]$.  Define $\pi'_q(n)$ to be the number of irreducible monic polynomials in $\M_q$ of degree exactly $n$. From an exact formula for $\pi'_q(n)$ that Gauss derived in 1797, it follows that
\begin{equation} \pi'_q(n) \sim \frac{q^n}{n}, \label{eq:ffpnt} \end{equation}
which can be regarded as a function field analogue of the Prime Number Theorem.  In fact it follows easily from Gauss' formula that we always have the bound
\[ \pi'_q(n) \leq \frac{q^n}{n}\]
which will be used frequently below.  

Here, we investigate the size of the $k$-th irreducible polynomial if we impose a total ordering on the irreducible polynomials, analogous to that which exists for the primes. Because the irreducibles are already partially ordered by degree, we will require that our ordering respects degree. We let $P_k$ denote the $k$-th irreducible polynomial in $\M_q$ under any such ordering and find bounds for both the degree and norm of $P_k$.

\begin{theorem} \label{thm:kthirred}
If $\{P_k\}$ is an arbitrary ordering of irreducible polynomials of increasing degree in $\M_q$, then as $k\to \infty$ we can bound the degree of $P_k$ by 
\[\log_q k + \log_q(\log_q k) + \log_q(q-1) -1 +  o(1) \ \leq \ \deg P_k \ \leq \ \log_q k + \log_q(\log_q k) + \log_q(q-1) +  o(1)\]
where $\log_q$ denotes the logarithm base $q$.

\end{theorem}

\begin{proof}
We know that 
\[\pi_q(\deg P_k - 1) < k \leq \pi_q(\deg P_k),\]
where $\pi_q(n)$ is the counting function for monic irreducible polynomials of degree up to $n$. The function $\pi_q(n)$ is more difficult to work with than $\pi'_q(n)$, but Kruse and Stichtenoth \cite{kruse} have obtained an asymptotic expression for this quantity:
\[\pi_q(n) \sim \frac{q^{n+1}}{(q-1)n}.\]
Thus as $k\to \infty$,
\begin{equation}
    \frac{q^{\deg P_k}}{(q - 1)(\deg P_k - 1)}(1+o(1)) \leq k \leq \frac{q^{\deg P_k+1}}{(q-1)\deg P_k}(1+o(1)). \label{eq:k-bounds}
\end{equation}
Taking the base-$q$ logarithm of both sides of the left inequality gives
\begin{align*}
    \deg P_k &\leq \log_q k + \log_q(\deg P_k - 1) +\log_q(q-1) +  o(1) \\
    & \leq \log_q k + \log_q(\log_q k) + \log_q(q-1) +  o(1)
\end{align*}
where the upper bound for $\deg P_k$ was substituted into the expression to obtain the second line.

Likewise taking logs of the terms forming the right inequality of \eqref{eq:k-bounds} gives 
\begin{align*}
    \deg P_k & \geq \log_q k + \log_q(\deg P_k) +\log_q(q-1)-1 +  o(1) \\
    & \geq \log_q k + \log_q(\log_q k) + \log_q(q-1) -1 +  o(1).
\end{align*}
Again, the lower bound for $\deg P_k$ was substituted into the expression to obtain the second line.
\end{proof}

\begin{corollary}\label{cor:ordered polynomials}
If $\{P_k\}$ is an arbitrary ordering of irreducible polynomials of increasing degree in $\M_q$, then \[\deg P_k = \log_q k + \log_q(\log k) +  O(1)\]
and
\[\|P_k\| \asymp_q k \log k. \]
\end{corollary}

\begin{remark}
This result is essentially best possible, since we know that the degree of the $k$-th irreducible will jump by 1 every time the irreducibles of a given degree are exhausted (and thus the norm will increase by a factor of $q$.)  For comparison, over the integers it is known \cite{Cipolla02} that \[\log p_n = \log n + \log \log n + \frac{\log \log n}{\log n} − \frac{1}{\log n} - \frac{(\log \log n)^2-2\log \log n+5}{2\log^2 n} +O\left(\left(\frac{\log \log n}{\log n}\right)^3\right)
\]
and that 
\[p_n = n\left(\log n +\log \log n -1 +\frac{\log \log n -2}{\log n} + O\left(\left(\frac{\log \log n}{\log n}\right)^2\right)\right)\]
where $p_n$ is the $n$-th prime number.
\end{remark}

\section{Polynomials with $k$ irreducible factors}
Having investigated the distribution of irreducibles in the function field, a natural extension is to consider the distribution of monic polynomials with $k$ irreducible factors. To this end, we introduce the counting function $\Pi'_{q,k}(n)$, which denotes the number of squarefree monic polynomials of degree $n$ with $k$ irreducible factors. Here, we investigate an asymptotic formula for $\Pi'_{q,k}(n)$ and determine a strict upper bound that gives us a quantitative analogue of the Hardy-Ramanujan theorem for function fields. 
\subsection{The Sathe-Selberg Formula}\label{sect.ss}
Afshar and Porrit \cite{Porritt} recently showed that an analogue of the Sathe-Selberg theorem holds for function fields, which gives us an asymptotic expression for $\Pi'_{q,k}(n)$.
\begin{theorem}\label{thm.ss}
Let $C \in [1,2]$. For $n ≥ 2$ and $1 ≤ k ≤ C\log n + 1$,
\[\Pi'_{q,k}(n)= \frac{q^n}{n} \cdot  \frac{\log^{k-1}n}{(k-1)!} \left( G\left(\frac{k-1}{\log n}\right) + O_C\left(\frac{k}{\log ^2 n}\right) \right),\]
where
\[G(z) = \frac{1}{\Gamma(z+1)} \prod_{p \in \I} \left(1 + \frac{z}{\|p\|}\right) \left(1 - \frac{1}{\|p\|}\right)^z.\]
\end{theorem}

We can obtain a simplified version of this asymptotic by showing that $G(\frac{k-1}{\log n})$ is bounded away from zero and infinity in this range.
\begin{lemma} \label{lem.ssbounded}
In the range $0\leq z \leq 2$ we have
\[e^{\m3} ≤ G\left(z\right) ≤ 1.\]
\end{lemma}
\begin{proof}
Note that \[\frac{d}{dz} \left(1 + \frac{z}{\|p\|}\right) \left(1 - \frac{1}{\|p\|}\right)^z =\left(1- \frac{1}{\|p\|}\right)^z\left(\left(1+\frac{z}{\|p\|}\right)\log\left(1-\frac{1}{\|p\|}\right) + \frac{1}{\|p\|}\right)<0\]
for all $z>0$.  Since $\frac{1}{\Gamma(z +1)}$ is also decreasing on this range, we can bound $G(2)\leq G(z) \leq G(0)$ for all $z$ in this ranhaveupper bound is then obtained by noting that $G(0)=1$.  We obtain the lower bound by estimating $G(2)$.

\begin{align*}
    G(2) &= \frac{1}{\Gamma(3)}\prod_{p \in \I} \left(1 + \frac{2}{\|p\|}\right) \left(1 - \frac{1}{\|p\|}\right)^2\\
    &=\frac{1}{2}\prod_{n=1}^\infty \left( \left(1 + \frac{2}{q^{n}}\right) \left(1 - \frac{1}{q^{n}}\right)^2 \right)^{\pi'(n)}\\
    &\geq \frac{1}{2}\prod_{n=1}^\infty \left( 1 - \frac{3}{q^{2n}}+\frac{2}{q^{3n}} \right)^{\frac{q^n}{n}}.
\end{align*}
Taking logs, and using that $\log(1-x)\geq -x-x^2$ for $0<x\leq 1/2$ this expression is
\begin{align*}
    -\log 2 &+\sum_{n=1}^\infty \frac{q^n}{n}\log\left(1- \frac{3}{q^{2n}}+\frac{2}{q^{3n}}  \right) > -\log 2 -\sum_{n=1}^\infty \frac{q^n}{n}\left( \left(\frac{3}{q^{2n}}-\frac{2}{q^{3n}}\right) + \left(\frac{3}{q^{2n}}-\frac{2}{q^{3n}}\right)^2\right) \\
    &=-\log 2 -\sum_{n=1}^\infty \left( \frac{3}{nq^{n}}-\frac{2}{nq^{2n}} + \frac{9}{nq^{3n}}-\frac{12}{nq^{4n}}+\frac{4}{nq^{5n}}\right)\\
    &=-\log 2 +3\log\left(1{-}\frac{1}{q}\right)-2\log\left(1{-}\frac{1}{q^2}\right)+9\log\left(1{-}\frac{1}{q^3}\right)-12\log\left(1{-}\frac{1}{q^4}\right)+4\log\left(1{-}\frac{1}{q^5}\right)\\
    &\geq-\log 2 +3\log\left(\frac{1}{2}\right)-2\log\left(\frac{3}{4}\right)+9\log\left(\frac{7}{8}\right)-12\log\left(\frac{15}{16}\right)+4\log\left(\frac{31}{32}\right)>-2.75
\end{align*}
Thus $G(2)>e^{-3}$ as claimed.
\end{proof}

Because $G(\frac{k-1}{\log n})$ is bounded away from zero, we can write our asymptotic for $\Pi'_{q,k}(n)$ in a more convenient form. If we define 
\[H_k(n) = \frac{q^n}{n} \cdot \frac{\log^{k-1} n}{(k-1)!} G\left(\frac{k-1}{\log n}\right),\]
then we can write Theorem \ref{thm.ss} in the form 
\begin{equation} \Pi'_{q,k}(n) = H_k(n) \left(1 + O_C\left(\frac{k}{\log^2 n}\right) \right), \label{eq:SSH} \end{equation}
which we will find useful for later sections.

\subsection{The Hardy-Ramanjuan Inequality}
The Sathe-Selberg formula implies that 
\[\Pi'_{q,k}(n) = O\left(\frac{q^n\log^{k-1}n}{n(k-1)!}\right)\]
for all $k \in [1, C\log n + 1]$. In this section, we obtain a uniform upper bound of this form, valid for all $k$ and $n$, and use it to obtain bounds for the number of monic polynomials with at most $\alpha\log n$ or at least $\beta\log n$ prime factors. Our result can be interpreted as a quantitative version of the Hardy-Ramanujan theorem for function fields, that almost all degree $n$ polynomials have about $\log n$ distinct prime factors. 
\begin{lemma}\label{lem.piasymp}
\[(k-1)\Pi'_{q,k}(n) ≤ \sum_{\mathclap{\substack{p \in \I \\ \deg p \leq n/2}}} \Pi'_{q,k-1}(n-\deg p). \]
\end{lemma}
\begin{proof}
Each polynomial counted by $\Pi'_{q,k}(n)$ has the form $p_1p_2\ldots p_k$, where the $p_i$ are distinct irreducibles whose degrees add to $n$. At least $k-1$ of the $p_i$ have degree less than $n/2$. If we fix an irreducible $p$ with degree less than $n/2$, we can multiply it with a squarefree polynomial of degree $n- \deg p$ with $k-1$ irreducible factors to obtain a degree $n$ polynomial with $k$ irreducible factors. This polynomial may no longer be squarefree, but this is acceptable since we are only looking for an upper bound. 

For each choice of $p$, we obtain $\Pi'_{q,k-1}(n - \deg p)$ such polynomials with $k$ prime factors. Summing over all irreducibles $p$ of degree at most $n/2$ gives us 
\[\sum_{\mathclap{\substack{p \in \I \\ \deg p \leq n/2}}} \Pi'_{q,k-1}(n-\deg p).\]
This expression overcounts $\Pi'_{q,k}(n)$ by at least a factor of $k-1$, since there are at least $k-1$ choices for which factor $p$ was used to construct a given such polynomial with $k$ prime factors. Hence, dividing it by $k-1$ gives the desired upper bound for $\Pi'_{q,k}(n)$.
\end{proof}

We now derive an explicit upper bound for the number of squarefree polynomials of degree $n$ having exactly $k$ factors.  Similar, explicit bounds have been given over the integers, see for example Theorem 3.5 of \cite{bkk19}.
\begin{theorem}\label{thm:quant-hr}
\[\Pi'_{q,k}(n) \leq \frac{q^n}{n}\frac{(\log n + c)^{k-1}}{(k-1)!}\]
for all $k$ and $n$, where $c = 2 - \log 2 = 1.3068\ldots$
\end{theorem}
\begin{proof}
We establish the claim by induction on $k$. When $k = 1$, $\Pi'_{q,k}(n)$ counts the number of monic irreducibles of degree $n$, so $\Pi'_1(n) = \pi'_q(n) \leq \frac{q^n}{n}$. 

Now assume the claim is true for $k = j$, so that 
\[\Pi'_{q,j}(n) \leq \frac{q^n}{n}\frac{(\log n + c)^{j-1}}{(j-1)!}.\]
Then by Lemma \ref{lem.piasymp},
\begin{align}
\Pi'_{q,j+1}(n) &\leq
\frac{1}{j}\sum_{\mathclap{\substack{ p \in \mathcal{I}_q \\ \deg p ≤ n/2}}}
\frac{q^{n- \deg p}}{n - \deg p}
\frac{(\log (n - \deg p)+c)^{j-1}}{(j-1)!} \nonumber \\
&\leq \frac{q^n(\log n + c)^{j-1}}{j!} \sum_{\substack{ p \in \mathcal{I}_q \\ \deg p ≤ n/2}} \frac{1}{(n - \deg p)q^{\deg p}}. \label{eq:pij_upperbound}
\end{align}
We now find an upper bound for this sum. Note that 
\begin{align*} 
\sum_{\substack{ p \in \mathcal{I}_q \\ \deg p ≤ n/2}} \frac{1}{(n - \deg p)q^{\deg p}} &= \sum_{\substack{ p \in \mathcal{I}_q \\ \deg p ≤ n/2}} \frac{1}{nq^{\deg p}}\cdot \frac{1}{1 - \frac{\deg p}{n}} \\
&= \sum_{\substack{ p \in \mathcal{I}_q \\ \deg p \leq n/2}} \frac{1}{nq^{\deg p}} \left(1 + \frac{\deg p}{n} + \left(\frac{\deg p}{n}\right)^2 + \ldots \right)\\
&= \sum_{\substack{ p \in \mathcal{I}_q \\ \deg p \leq n/2}} \left( \frac{1}{nq^{\deg p}} + \frac{\deg p}{n^2q^{\deg p}} \left(1 + \frac{\deg p}{n} + \cdots \right) \right)\\
&= \sum_{\substack{ p \in \mathcal{I}_q \\ \deg p \leq n/2}} \frac{1}{nq^{\deg p}} +\sum_{\substack{ p \in \mathcal{I}_q \\ \deg p \leq n/2}}  \left( \frac{\deg p}{n^2q^{\deg p}} \cdot \frac{1}{1 - \frac{\deg p}{n}} \right).
\end{align*}
We bound the first summation by noting that \[\sum_{\substack{ p \in \mathcal{I}_q \\ \deg p ≤ n/2}} \frac{1}{nq^{\deg p}} = \frac{1}{n}\sum_{k = 1}^{\floor{n/2}} \frac{\pi'_q(k)}{q^k} \leq \frac{1}{n}\sum_{k = 1}^{\floor{n/2}} \frac{1}{k} \leq \frac{\log n - \log 2 + 1}{n}. \]
For the second summation, note that $1/(1 - \frac{\deg p}{n}) \leq 2$ since $\deg p \leq n/2$, so
\[\sum_{\substack{ p \in \mathcal{I}_q \\ \deg p ≤ n/2}} \left( \frac{\deg p}{n^2 q^{\deg p}}\cdot \frac{1}{1 - \frac{\deg p}{n}} \right) \leq \frac{2}{n^2}\sum_{k = 1}^{\floor{n/2}} \frac{k \pi'_q(k)}{q^k} \leq \frac{2}{n^2} \sum_{k = 1}^{\floor{n/2}}1 \leq \frac{2}{n^2} \cdot \frac{n}{2} = \frac{1}{n}. \]
Hence 
\[\sum_{\substack{ p \in \mathcal{I}_q \\ \deg p ≤ n/2}} \frac{1}{(n - \deg p)q^{\deg p}} \leq \frac{\log n + 2 - \log 2}{n}. \]
Inserting this into \eqref{eq:pij_upperbound} we can conclude 
\[\Pi'_{q,j+1}(n) \leq \frac{q^n}{n}\frac{(\log n + c)^j}{j!}\]
as desired. 
\end{proof}

\begin{proposition}\label{prop.quant-hr-2}
Let $n$ be an integer and let $\alpha$ and $\beta$ be constants such that $0 < \alpha < 1 < \beta$. Then the number of polynomials of degree $n$ having less than $\alpha\log n$ or more than $\beta \log n$ prime factors satisfies the following bounds.
\[\sum_{k \leq \alpha \log n} \Pi'_{q,k}(n) \ll_\alpha
 \frac{q^n}{n^{Q(\alpha)}\sqrt{\log n}}  \hspace{2mm} \textnormal{and} \hspace{2mm} \sum_{k \geq \beta\log n} \Pi'_{q,k}(n) \ll_\beta
 \frac{q^n}{n^{Q(\beta)}\sqrt{\log n}} ,\]
 where $Q(y) = y\log y - y + 1$.
\end{proposition}
\begin{proof}
By Theorem \ref{thm:quant-hr}, 
\[\sum_{k \leq \alpha\log n} \Pi'_{q,k}(n) \leq \sum_{k \leq \lfloor \alpha\log n \rfloor} \frac{q^n(\log n+c)^{k-1}}{n(k-1)!}   \hspace{2mm} \textnormal{and} \hspace{2mm} \sum_{k \geq \beta\log n} \Pi'_{q,k}(n) \leq  \sum_{k \geq \beta \log n } \frac{q^n(\log n+c)^{k-1}}{n(k-1)!}.\]
When $x > 0$ and $0 < \alpha < 1 < \beta$, \cite{norton} gives us the bounds
\begin{equation}
    \sum_{k \leq \alpha x} \frac{e^{-x} x^k}{k!} < \frac{e^{-Q(\alpha) x}}{(1-\alpha)\sqrt{\alpha x}}\hspace{2mm} \textnormal{and} \hspace{2mm} \sum_{k \geq \beta x} \frac{e^{-x} x^k}{k!} < \frac{e^{-Q(\beta)x}}{{(\beta-1)\sqrt{2 \pi \beta x}}} , \label{eq:Norton}
\end{equation}
where  $Q(y) = y \log y - y + 1$. Letting $x = \log n+c$ in the first expression gives us 
\begin{align*}
    \sum_{k \leq \alpha\log n} \Pi'_{q,k}(n) \leq q^ne^c\sum_{k \leq  \alpha(\log n + c)} \frac{(\log n+c)^{k-1}}{e^{\log n +c}(k-1)!} <  \frac{q^n e^{c-Q(\alpha)(\log n + c)}}{(1-\alpha)\sqrt{\alpha(\log n +c)}} \ll_\alpha \frac{q^n}{n^{Q(\alpha)}\sqrt{\log n}}.
\end{align*}
For the second expression we have, using \eqref{eq:Norton} and Stirling's approximation
\begin{align*}
    \sum_{k \geq \beta\log n} \Pi'_{q,k}(n) &= \sum_{\beta \log n \leq  k < \beta(\log n+c)} \Pi'_{q,k}(n)+\sum_{k \geq \beta(\log n+c)} \Pi'_{q,k}(n)\\
    &< (c\beta+1)\frac{q^n(\log n+c)^{\lfloor\beta\log n\rfloor-1}}{n(\lfloor \beta \log n\rfloor-1)!} + q^ne^c\sum_{k \geq  \beta (\log n + c)} \frac{(\log n+c)^{k-1}}{e^{\log n +c}(k-1)!} \\
    &\ll_\beta \frac{q^n(\log n+c)^{\lfloor \beta\log n\rfloor-1}}{n\sqrt{\log n} ( (\beta \log n-1)/e)^{\lfloor \beta\log n\rfloor-1}} + \frac{q^n}{n^{Q(\beta)}\sqrt{\log n}}\\
    &\ll \frac{q^n}{n\sqrt{\log n}}\left(\frac{e}{\beta}\right)^{\lfloor \beta\log n\rfloor-1}+ \frac{q^n}{n^{Q(\beta)}\sqrt{\log n}} \ll \frac{q^n}{n^{Q(\beta)}\sqrt{\log n}}.\qedhere
\end{align*}
\end{proof}

\section{Primitive Sets with Consistently Growing Counting Functions}
Having considered the asymptotics of $\pi'_q(n)$ and $\Pi'_{q,k}(n)$, we are ready to construct primitive sets with consistently growing counting functions. We adapt a construction of Martin and Pomerance in the integers \cite{martin-pomerance} to prove the following theorem.

\begin{theorem}\label{thm.pom}
Suppose $L(x)$ is positive and increasing, that $L(x) \sim L(2x)$, and that
        \[\int_2^\infty \frac{dt}{t\log t \cdot L(t)} < \infty.\]
        Then there exists a primitive set $S \subset \M_q$ such that $S'(n)$ satisfies
        \[S'(n) \asymp \frac{q^n}{\log n \cdot \log \log n \cdot L(\log n)}.\]
\end{theorem}

Taking $L(x) = \log \log x \cdot \log \log \log x \cdots \big(\!\log_j x\big)^{1+\e}$ in Theorem \ref{thm.pom} for some $j\geq 2$ (here $\log_j(x)$ denotes the $j$-fold iterated logarithm) we have the following corollary. 

\begin{corollary}
For any $\e > 0$ and $j\geq 2$ there exists a primitive set $S \subset \M_q$ such that $S'(n)$ satisfies
    \[S'(n) \asymp \frac{q^n}{\log n \cdot \log \log n \cdots \big(\! \log_{j} n \big)^{1+\e} }.\] 
\end{corollary}

In a sense, this is the slowest-growing function that $L(x)$ could possibly be, since $\int_2^\infty \frac{dt}{t \log t \cdots (\log_j t)^{1+\e}}$ grows without bound as $\e$ tends to $0$. Thus, our corollary gives the fastest-growing asymptotic counting function achievable using Theorem \ref{thm.pom}. In particular, it is much larger than the counting function of the irreducible polynomials of degree $n$, which is asymptotic to $\frac{q^n}{n}$.

\subsection{Constructing a sequence of irreducible polynomials}
A critical ingredient in the construction of our primitive set $S$ will be a sequence of monic irreducible polynomials $\{t_k\}$ in $\M_q$. In order for $S'(n)$ to have the desired asymptotics, we need to impose two conditions this sequence $\{t_k\}$: $\sum_{i=1}^{\infty} \frac{1}{\|t_i\|} < \frac{1}{2}$ and $\|t_k\| \asymp kL(k)\log k$. The following proposition guarantees the existence of such a sequence. 

\begin{proposition}\label{prop.constructing p_k}
Let $L(x)$ be positive and increasing, such that $L(x) \sim L(2x)$ and 
\[ \int_2^\infty \frac{dt}{t\log tL(t)}\ < \infty. \]
Then there exists a sequence of irreducible polynomials $\{t_i\}$ of nondecreasing degree such that $\sum_{i=1}^{\infty} \frac{1}{\|t_i\|} < \frac{1}{2}$ and $\|t_k\| \asymp_q kL(k)\log k$.
\end{proposition}

\begin{proof}
Because $L(x)$ is increasing and $\int_2^\infty \frac{dt}{tL(t)}$ converges, there exists some integer $y_0$ such that $L(y) ≥ 1$ for all $y ≥ y_0$. Without loss of generality let $y_0 ≥ 3$. Fix an ordering of the irreducible polynomials that respects degree and consider the sequence of polynomials $\{r_k\}$ defined by
\[r_k = \left\{
\begin{array}{ll}
      \text{the $k$th irreducible polynomial} & \text{if } k < y_0 \\
      \text{the $\floor{kL(k)}$th irreducible polynomial} & \text{if } k ≥ y_0. \\
\end{array} \right.\]

For $k ≥ y_0$ we have that $\floor{(k+1)L(k+1)} ≥ \floor{(k+1)L(k)} \geq \floor{kL(k)+1} = \floor{kL(k)} + 1$, so the indices of the polynomials in $\{ r_k\}$ are strictly increasing for all $k$. By Corollary \ref{cor:ordered polynomials}, the degree of the $k$th irreducible polynomial is $\log_q (k) + \log_q(\log k) +O(1)$. Using this, we can find that when $k ≥ y_0$,

\vspace{-1.2em}
\[\|r_k\| = q^{\log_q(kL(k)) + \log_q(\log(kL(k)) ) +O(1)} = kL(k)\log(kL(k))q^{O(1)} \asymp_q  kL(k)\log(k).\]
Thus
\[\sum_{k=y_0}^\infty \frac{1}{\|r_k\|} \ll_q \sum_{k=y_0}^\infty \frac{1}{k L(k)\log k} ≤  \int_{2}^\infty \frac{dt}{t\log tL(t)}.\]
Since the last integral is bounded by assumption, the left hand sum must also be bounded. Thus there exists some $k_0 ≥ y_0$ such that $\sum_{k=k_0}^\infty \frac{1}{\|r_k\|} < \frac{1}{2}$.\\
\\
Then the sequence $\{t_k\}$ given by $t_k = r_{k_0+k}$ has the property that $\sum_{k=1}^\infty \frac{1}{\|t_k\|} < \frac{1}{2}$. Furthermore, by Corollary \ref{cor:ordered polynomials},
\[\|t_k\| \asymp \floor{(k-k_0) L(k-k_0)} \log (\floor{(k-k_0) L(k-k_0)}) \sim k L(k) \log k. \qedhere\]
\end{proof}

We now use the sequence $\{t_k\}$ from Proposition \ref{prop.constructing p_k} to construct a primitive set by defining
        \[S_k = \{f \in \mathcal{M}_q : f \textnormal{ squarefree, }   \omega(f) = k,\ t_k | f, \text{ and }  t_j \!\!\!\not| f\ \textnormal{for all } j < k\},\] 
where $\omega(f)$ denotes the number of irreducible factors of $f$. Then, we let $S = \bigcup_{k=1}^{\infty} S_k$.

\begin{proposition}
$S$ is a primitive set.
\end{proposition}

\begin{proof}
Each $S_k$ is primitive, since  any multiple of a polynomial with $k$ irreducible factors must have more than $k$ irreducible factors. Hence if we assume for the sake of contradiction that $S$ is not primitive, then there exist $f, g \in S$ with $f|g$ such that $f \in S_m$ and $g \in S_n$ where $m < n$. However this is impossible since $t_n |f$ but $t_n \!\!\not| g$.
\end{proof}

\subsection{Asymptotics of $S'_k(n)$}

\begin{lemma} \label{lem.pomasymp2}
Suppose $k \ll \log n$ and the sequence of polynomials $\{t_k\}$ is as defined above.  Then
\[\frac{k-1}{\log^2(n-\deg t_k)} = O\left(\frac{1}{\log n}\right).\]
\end{lemma}
\begin{proof}
Since $\deg t_k \ll \log k \ll \log \log n$,
\[\frac{k-1}{\log^2(n-\deg t_k)} \ll  \frac{\log \log  n}{ \log^2 n} \ll \frac{1}{\log n}. \qedhere \]
\end{proof}
Note that this result also holds if $\deg t_k$ is replaced with $\deg t_k + \deg t_j$ for some $j < k$ because $\deg t_k + \deg t_j = O(\deg t_k)$.

\begin{theorem}\label{thm.5}
For sufficiently large $n$ and $k \in [2, \frac{3}{2}\log n]$,
\[S_k'(n) \asymp \frac{q^{n-\deg t_k}}{n - \deg t_k}\cdot \frac{\log (n - \deg t_k)^{k-2}}{(k-2)!}.\]
\end{theorem} 

\begin{proof}
We can bound $S'_k(n)$ by  $\Pi'_{q,k}(n)$ as
\[\Pi'_{q,k-1}(n-\deg t_k) \  ≥  \ S_k'(n) \ ≥ \ \Pi'_{q,k-1}(n - \deg t_k) - \sum_{j=1}^{k-1} \Pi'_{q,k-2}(n - \deg t_j - \deg t_k).\]
Here, the upper bound comes from the observation that $S_k'(n)$ counts a subset of the squarefree polynomials $f$ such that $\omega(f) = k$ and $t_k \mid f$, while the lower bound is obtained by removing from this set the multiples of $t_j$ for all $j < k$.  For sufficiently large $n$, we can apply Sathe-Selberg for function fields in the form of \eqref{eq:SSH} to get
\[S_k'(n) \leq H_{k-1}(n - \deg t_k) \bigg(1 + O\bigg(\frac{k-1}{\log^2(n - \deg t_k)}\bigg)\bigg) = H_{k-1} (n - \deg t_k) \Bigg(1 + O\Bigg(\frac{1}{\log n}\Bigg)\Bigg), \]
and 
\begin{align}
S_k'(n) &\geq
 H_{k-1} (n - \deg t_k) \bigg(1 + O\bigg(\frac{k-1}{\log^2(n - \deg t_k)}\bigg)\bigg) \nonumber \\
 & \hspace{3cm} -\sum_{j=1}^{k-1}  H_{k-2} (n - \deg t_k - \deg t_j)\Bigg(1 + O\bigg(\frac{k-1}{\log^2(n - \deg t_k - \deg t_j)}\bigg)\Bigg) \nonumber \\
 &=\Bigg( H_{k-1} (n -\deg t_k) - \sum_{j=1}^{k-1} H_{k-2} (n - \deg t_k - \deg t_j) \Bigg) \Bigg( 1 + O\Bigg(\frac{1}{\log n} \Bigg) \Bigg). \label{eq:skprimelowbd}
 \end{align}
Here $H_k(n)$ is the function defined in Section \ref{sect.ss}, and Lemma \ref{lem.pomasymp2} has been used to simplify the error terms.  Note that the upper bound already gives 
\[S_k'(n) \ll \frac{q^{n-\deg t_k}}{n - \deg t_k}\cdot \frac{\log (n - \deg t_k)^{k-2}}{(k-2)!} \]
(using that $G(z) =O(1)$) so we need only concern ourselves with the lower bound. 

We now factor out $H_{k-1} (n - \deg t_k)$ from each term of \eqref{eq:skprimelowbd} and consider the ratio 
\begin{align}
    &\hspace{-10mm}\frac{\sum_{j=1}^{k-1} H_{k-2} (n - \deg t_k - \deg t_j)}{H_{k-1} (n - \deg t_k)} \nonumber \\
    &= (k-2) \sum_{j=1}^{k-1} \left(\frac{G\left(\frac{k-3}{\log(n-\deg t_k -\deg t_j)}\right)}{\|t_j\|G\left(\frac{k-2}{\log(n-\deg t_k )}\right)} \frac{n - \deg t_k}{n - \deg t_k -\deg t_j} \frac{\log^{k-3} (n - \deg t_k - \deg t_j)}{\log^{k-2}(n - \deg t_k)}\right) \nonumber \\
    &< \frac{k}{\log(n-\deg t_k)} \sum_{j=1}^{k-1} \left(\frac{G\left(\frac{k-3}{\log(n-\deg t_k -\deg t_j)}\right)}{\|t_j\|G\left(\frac{k-2}{\log(n-\deg t_k )}\right)} \frac{n - \deg t_k}{n - \deg t_k -\deg t_j} \right)\nonumber \\
    &= \frac{k}{\log n} \sum_{j=1}^{k-1} \left(\frac{G\left(\frac{k-3}{\log(n-\deg t_k -\deg t_j)}\right)}{\|t_j\|G\left(\frac{k-2}{\log(n-\deg t_k )}\right)} \right)\left(1+O\left(\frac{\log n}{n}\right)\right). \label{eq:Hratioapprox}
\end{align}
Here we have used that \[n-\deg t_k = n+O(\log k) = n + O(\log \log n).\]
Because the function $G(z)$, defined in Theorem \ref{thm.ss}, is analytic, and its derivative is bounded in the interval $[0,2]$, we can write \begin{align*}
    G\left(\frac{k - 3}{\log (n - \deg t_k - \deg t_j)}\right) &= G\left(\frac{k - 2}{\log (n - \deg t_k) } + O\left(\frac{1}{\log n}\right)\right) \\
    &= G\left(\frac{k - 2}{\log (n - \deg t_k) }\right)\left(1 + O\left(\frac{1}{\log n}\right)\right).
\end{align*}
Inserting this into \eqref{eq:Hratioapprox} we find that \[\frac{\sum_{j=1}^{k-1} H_{k-2} (n - \deg t_k - \deg t_j)}{H_{k-1} (n - \deg t_k)} \leq \frac{k}{\log n } \left(\sum_{j=1}^{k-1} \frac{1}{\|t_j\|} \right)\left(1+O\left(\frac{1}{\log n}\right)\right).\] 
\noindent Using the expression above in \eqref{eq:skprimelowbd}, we see that 
\[ S_k'(n) \geq H_{k-1} (n - \deg t_k ) \Bigg( 1 - \frac{k}{\log n}\sum_{j=1}^{k-1} \frac{1}{\|t_j\|} \Bigg) \Bigg( 1 + O\Big(\frac{1}{\log n }\Big)\Bigg).\]
Because we have chosen the $t_i$ with $\sum_{j =1}^{k-1} \frac{1}{\|t_j\|} < \frac{1}{2}$, and $k \leq \frac{3}{2}\log n$ we have ${1-\frac{k}{\log n}\sum_{j =1}^{k-1} \frac{1}{\|\deg t_j\|} >\frac{1}{4}}$. This allows us to conclude that 
\[S_k'(n) \asymp H_{k-1} (n - \deg t_k) \asymp \frac{q^{n-\deg t_k}}{n - \deg t_k}\cdot \frac{\log^{k-2}(n - \deg t_k)}{(k-2)!}. \qedhere\]
\end{proof}
\subsection{The size of $S'(n)$}

\begin{proof}[Proof of Theorem \ref{thm.pom}]

Let $B = B(n) = \lfloor \frac{1}{2} \log n \rfloor$ and $B' = B'(n) = \lfloor \frac{3}{2} \log n \rfloor$. When $n \geq \deg t_1$, we show that 
\begin{equation} q^{n - \deg t_B} \gg S'(n) \gg q^{n - \deg t_{B'}}. \label{eq:sprimebounds} \end{equation}

Because $S$ is a disjoint union of the sets $S_k$, we can bound $S'(n)$ using our bounds for $S'_k(n)$ from Theorem \ref{thm.5}. As a lower bound, we have
$$S'(n)  \geq \sum_{k = 1}^{B'} S'_k(n) \gg \sum_{k=1}^{B'} \frac{q^{n -\deg t_k}}{n - \deg t_k}\cdot \frac{\log^{k-2}(n - \deg t_k)}{(k-2)!} \geq \frac{q^{n - \deg t_{B'}}}{n} \sum_{k =1}^{B'} \frac{\log^{k-2}(n - \deg t_{B'})}{(k-2)!}.$$
Since $\sum_{j = 0}^{\lfloor y\rfloor}  \frac{y^j}{j!} \gg e^y$, we get 
$$ \sum_{k =2}^{B'} \frac{\log^{k-2}(n - \deg t_{B'})}{(k-2)!} \gg n - \deg t_{B'},$$
which implies $S'(n) \gg q^{n - \deg t_{B'}}$. Similarly, we can bound $S'(n)$ from above:
$$S'(n) \leq \sum_{k=1}^{\infty} S'_k(n) \leq \sum_{k = B+1}^{B'} S'_k(n) + \sum_{k \leq \lfloor 1/2 \log n \rfloor} \Pi'_{q,k}(n)  + \sum_{k \geq \lfloor 3/2 \log n \rfloor} \Pi'_{q,k}(n).$$
From Proposition \ref{prop.quant-hr-2}, we know that the latter two sums are bounded by $O\left(
 \frac{q^n}{n^{Q(1/2)}}\right)$ and $O\left(
 \frac{q^n}{n^{Q(3/2)}} \right)$ respectively, where $Q(\frac{1}{2}), Q(\frac{3}{2}) > 0$. We can bound the first sum by Theorem \ref{thm.5}:
$$\sum_{k = B+1}^{B'} S'_k(n) \ll  \sum_{k = B+1}^{B'} \frac{q^{n-\deg t_k}}{n -\deg t_k} \frac{\log^{k-2}n}{(k-2)!} \leq \frac{q^{n - \deg t_B}}{n - \deg t_B} \sum_{j =0}^{\infty}\frac{\log^j n}{j! \log n} = \frac{q^{m - \deg t_B}}{n - \deg t_B} \leq q^{n - \deg t_B}.$$
Since this term grows faster than $O(\frac{q^n}{n^c})$ for any positive constant $c$, we can disregard the other two sums and conclude that $S'(n) \ll q^{n - \deg t_B}$.

Now, recall that our sequence of polynomials $\{t_k\}$ satisfies 
\[\|t_k\| \asymp kL(k)\log k;\]
letting $k = c\log n$, we have 
\[\|t_k\| \asymp c\log n \cdot L(c\log n) \cdot \log(c\log n) \sim c\log n \cdot \log \log n \cdot L(\log n).\]
From \eqref{eq:sprimebounds} we have 
\[S'(n) \gg q^{n-\deg t_B} = \frac{q^n}{\|t_{\lfloor \frac{1}{2} \log n\rfloor}\|} \gg \frac{q^n}{\log n \cdot \log \log n \cdot L(\log n)},\]
and similarly  
\[S'(n) \ll  q^{n-\deg t_{B'}} = \frac{q^n}{\|t_{\lfloor \frac{3}{2} \log n\rfloor}\|} \ll \frac{q^n}{\log n \cdot \log \log n \cdot L(\log n)},\]
proving the theorem.
\end{proof}

\section{Future Research}
Over the integers the Erd\H{o}s sum \eqref{eq:erdsum} has been used as a metric to compare the relative size of primitive sets.  In 1988, Erd\H{o}s proposed that the primes are, in a sense, the ``largest" primitive set under this metric. In particular, he made the following conjecture, which has attracted significant recent interest \cites{banksMartin13,erdzhang,lichtmanPom19,zhang91} but remains open: 
\begin{Conjecture}\textnormal{(Erd\H{o}s)} 
Let $A \subset \N$ be a primitive set, $A\neq \{1\}$ and $\P$ be the primes. Then, 
$$\sum_{a \in A}\frac{1}{a \log a} \leq \sum_{p \in \P} \frac{1}{p \log p}$$
\end{Conjecture}
We believe that the analogous conjecture holds in the function field case. 
\begin{Conjecture}
Let $A \subset \M_q$ be a primitive set, $A\neq \{1\}$ and $\I$ be the irreducible polynomials. Then,
$$\sum_{a \in A}\frac{1}{\|a\| \deg a} \leq \sum_{f \in \I} \frac{1}{\|f\| \deg f}.$$

\end{Conjecture}
We will further investigate the size of this sum for primitive subsets of $\F_q[x]$ in a future paper.

\section*{Acknowledgements} This research was conducted as part of the SUMRY (Summer Undergraduate Mathematics Research at Yale) program in Summer 2019.  We would also like to thank the anonymous referee for helpful suggestions.

\renewcommand{\biblistfont}{\normalfont\normalsize}
\bibliographystyle{amsplain}
\bibliography{sources}

\end{document}